\documentclass[a4paper,11pt]{amsart}
\usepackage{amsmath,amssymb,amsfonts}
\usepackage{algorithmic}
\usepackage{algorithm}

\usepackage{amssymb}
\usepackage{amsfonts}

\newtheorem{theorem}{Theorem}

\newtheorem{claim}{Claim}

\newtheorem{conjecture}{Conjecture}
\begin{document}
\title[Nonrepetitive sequences]{Nonrepetitive sequences on arithmetic
progressions}
%
%
\author{ Jaros\L aw Grytczuk}
\address{
Department of Theoretical Computer Science, Jagiellonian University, Krak\'ow,
Poland and Faculty of Mathematics and Information Science, Warsaw University of
Technology, Warszawa, Poland}
\email{grytczuk@tcs.uj.edu.pl}

\author{Jakub Kozik}

\address{
Department of Theoretical Computer Science, Jagiellonian University, Krak\'ow,
Poland}
\email{jkozik@tcs.uj.edu.pl}

\author{Marcin Witkowski}

\address{Faculty of  Mathematics and Computer Science,
Adam Mickiewicz University, Pozna\'n, Poland}
\email{mw@amu.edu.pl}


\begin{abstract}
A sequence $S=s_{1}s_{2}\ldots s_{n}$ is \emph{nonrepetitive} if no two adjacent blocks of $S$ are identical. In 1906 Thue proved that there exist arbitrarily long nonrepetitive sequences over $3$-element set of symbols. We study a generalization of nonrepetitive sequences involving arithmetic progressions. We prove that for every $k\geqslant 1$, there exist arbitrarily long sequences over at most $2k+10 \sqrt{k}$ symbols whose subsequences, indexed by arithmetic progressions with common differences from the set $\{1,2,\ldots ,k\}$, are nonrepetitive. This improves a previous bound obtained in \cite{Grytczuk Rainbow}. Our approach is based on a technique introduced recently in \cite{GrytczukKozikMicek}, which was originally inspired by a constructive proof of the Lov\'{a}sz Local Lemma due to Moser and Tardos \cite{MoserTardos}. We also discuss some related problems that can be  attacked by this method.
\end{abstract}

\keywords{nonrepetitive sequence, randomized algorithm, the Lov\'{a}sz Local
Lemma.}
\thanks{Research of J. Grytczuk is supported by the Polish Ministry of Science and Higher Education grant (MNiSW) (N N206 257035). \\
Research of J. Kozik is supported by the Polish Ministry of Science and Higher Education grant (MNiSW) (N206 3761 375).}
\maketitle

\section{Introduction}

For a sequence $S=s_{1}s_{2}\ldots s_{n}$ a \emph{repetition} of
size $h$ is a block (subsequence of consequtive terms) of the form  $XX=x_{1}\ldots x_{h}x_{1}\ldots x_{h}$.
A sequence is \emph{nonrepetitive} if it does not contain a
repetition of any size $h\geqslant 1$. For example, the sequence $1231312$
contains a repetition $3131$ of size two, while $123132123$ is nonrepetitive. 

It is easy to see that the longest nonrepetitive sequence, which can be constructed over a set of two symbols, has length three.
In 1906 Thue \cite{Thue1}
proved, by a remarkable inductive construction, that there exist arbitrarily
long nonrepetitive sequences over just three different symbols (see also 
\cite{BerstelTrans}, \cite{BerstelWork}). This discovery resulted in many 
unexpected applications inspiring a stream of research and
leading to the emergence of new branches of mathematics with a variety of
challenging open problems (see \cite{AlloucheShallit}, \cite{Beanetal}, \cite%
{CurriePattern}, \cite{GrytczukMiki}, \cite{Lothaire1}).

One particular variant, proposed in \cite{Beanetal}, concerns \emph{%
nonrepetitive tilings}, i.e., assignments of symbols to lattice points of
the plane so that all lines in prescribed directions are nonrepetitive. This idea led Currie and Simpson \cite{CurrieSimpsson}\ to consider sequences
with a stronger property: all subsequences taken over arithmetic
progressions of bounded common differences are nonrepetitive. Let $%
k\geqslant 1$ be a fixed positive integer and let $S(k)$ be the family of
subsequences of $S$ of the form $s_{i}s_{i+d}s_{i+2d}\ldots s_{i+td}$ with $%
d\in \{1,2,\ldots ,k\}$, $1\leqslant i\leqslant k-1$, $t=\left\lfloor
n/d\right\rfloor $. If every element of $S(k)$ is a nonrepetitive sequence,
then $S$ is called \emph{nonrepetitive up to }$\bmod \ k$ (see \cite%
{CurrieSimpsson}). Let $M(k)$ denote the minimal number of symbols needed to
create arbitrarily long sequences nonrepetitive up to $\bmod \ k$. 
Thue's theorem can be rephrased as $M(1)=3$. It is easy to see that $%
M(k)\geqslant k+2$ for every $k\geqslant 1$, and one may suspect that
equality always holds.

\begin{conjecture}
$M(k)=k+2$ for every $k\geqslant 1$.
\end{conjecture}

This conjecture has been confirmed so far only for $k=2$, $3$, and $5$ (it
is not even  known for $k=4$) by providing Thue type constructions of
the desired sequences. However, using the Lov\'{a}sz Local Lemma (see \cite%
{AlonSpencer}) it was proved in \cite{Grytczuk Rainbow} that $M(k)\leqslant
e^{33}k$  for any $k$. In this paper we improve the last bound
substantially by proving that $M(k)\leqslant 2 k+O(\sqrt{k})$. Our method is
inspired by the recent constructive proof of the Lov\'{a}sz Local Lemma
due to Moser and Tardos \cite{MoserTardos}. Just like in the related paper \cite%
{GrytczukKozikMicek} (and in the original nonconstructive approach),
we prove the result for a more general version where symbols are chosen from prescribed
lists (sets) assigned to the positions in a sequence.  The same method applies in the case when $K$ is any $k$-element set positive integers, and we want to construct  arbitrarily long sequence with no repetitions on arithmetic subsequences with  differences from $K$.

\section{The algorithm}

We present an algorithm that generates consecutive terms of a sequence $S$
by choosing symbols at random (uniformly and independently), and every time
a repetition occurs, it erases the longest repeated block
and continues from the smallest unassigned position. We alway erase the 
block that contains the last chosen element in order to ensure that after this removal the remaining sequence stays nonrepetitive. In the listing of the algorithms value 0 of some $s_i$ means that no symbol is assigned to $s_i$. Initially all $s_i$ equals 0.

\begin{algorithm}
\label{alg}
\caption{Choosing a sequence which is non-repetitive up to mod $k$.}
\begin{algorithmic}[1]
\STATE $i \gets 1$
\WHILE{$i \leqslant n$}
  \STATE $s_i \gets$ random element of $L_i \backslash \{s_{i-k},s_{i-k+1},...,s_{i-1},s_{i+1},...,s_{i+k} \}$ 
	\IF {$s_1,...,s_n$ is non-repetitive with respect to non-zero elements} 
  	      \STATE $i\gets$ smallest index $j$ for which $s_j=0$
	\ELSE
  	      \STATE from the set of the longest repetitions in $S$ choose
  	      $$s_{j-2h \cdot d+d},...,s_{j-h \cdot d},s_{j-h \cdot d+d},...,s_j$$
  	       with the largest index of the first element $j-2h \cdot d +d$
  	      \IF {$i \leqslant j-h \cdot d$}
  	        \STATE $m \gets j-2h \cdot d+d$
  	      \ELSE
  	        \STATE $m \gets j-h\cdot d+d$    
  	      \ENDIF
  	      \FOR{$j = 1$ \TO $h$} 
							\STATE $s_m \gets 0$
							\STATE $m \gets m+d$
					\ENDFOR					     
  \STATE $i\gets$ smallest index $j$ for which $s_j=0$
	\ENDIF
\ENDWHILE 
\end{algorithmic}
\end{algorithm}

We show that for any given positive integer $n$, and arbitrary lists of symbols $L_{i}$, each of size at least $%
2k+10\sqrt{k}$, the Algorithm computes a sequence of
length $n$ which is nonrepetitive up to mod $k$. Random
elements in line (3) of the Algorithm are chosen independently with uniform distribution. The general idea is to prove that the Algorithm cannot work forever for all possible evaluations of the random experiments. It is easy to see that the Algorithm stops only if nonrepetitive up to $\bmod$ $k$ sequence is constructed.

\begin{theorem}
\label{Theorem 1}For every positive integer $n$, and for every sequence of sets $%
L_{1},\ldots ,L_{n}$, each of size at least $2 k+10 \sqrt{k}$, there is a
sequence $S=s_{1}\ldots s_{n}$ nonrepetitive up to $\bmod$ $k$ such
that $s_{i}\in L_{i}$ for every $i=1,2,\ldots ,n$.
\end{theorem}

\begin{proof}
Let us suppose for a contradiction that such sequence does not exists. It means that the Algorithm never stops.
We are going to count the possible sequences of random  values used in line
(3) of the algorithm in two ways. 


Let $r_{j}$, $1\leqslant j\leqslant M$, be a sequence of values chosen in the line (3) in the first $M$ choices of some run of the Algorithm . Each $r_j$ can take at least  $10 \sqrt{k}$ values. It means that there are at least $(10k)^{M/2}$ such sequences. 

The second way of counting involves descriptions of the behaviour of the Algorithm. For every fixed evaluation of the first $M$ random choices we define the following five elements:

\begin{itemize}
\item {A \emph{route $R$} on the upper right quadrant of a grid $\mathbb{Z}\times \mathbb{Z}$ from coordinate $(0,
0)$ to coordinate $(2M, 0)$ on $2M$ steps with possible moves $(1,1)$ and $%
(1,-1)$ which never goes below the axis $y=0$.} 

\item {A sequence $D$ of numbers between $1$ and $k$ corresponding to the
peaks on the route $R$, where by a \emph{peak} we mean a move $(1,1)$ followed
immediately by a move $(1,-1)$.}

\item {A sequence $O$ of numbers $-1$, or $1$ corresponding to the peaks on
route $R$.} 

\item {A sequence $P$ of integers, one for every peak, whose sum is not grater than $M$.} 

\item {A sequence $S$ produced by the Algorithm after $M$ steps.}
\end{itemize}

A pentad $(R,D,O,P,S)$ will be called a\textit{\ }\emph{log}. We
encode consecutive steps of the Algorithm into log in the following way:

Each time the algorithm executes line (3) we append a move $(1,1)$ to the route $R$ and for every execution of line (14)  we append $(1,-1)$. Notice that in line (14) the algorithm can set zero only to $%
s_{c}$ which are non-zero, therefore the number of down-steps on route $R$
never excess the number of up-steps, and it never goes below axis $y=0$. At
the end of computations we add to the route $R$ one down-step for each
element of $S$ which is non-zero. This brings us to the point $(2M,0)$.
Whenever Algorithm 1 executes line (7) we append to the sequence $D$ a
difference $d$ of the chosen longest repetition. Then, if (8) is true, we
append $1$ to the sequence $O$, otherwise we append $-1$ . 
For each execution of the loop (13)-(16) we append to the the sequence $P$,  value $j$ for which $m$ equals $i$ in the loop. Finally, $S$ is the sequence  produced by the Algorithm  after $M$ executions of line (3).

\begin{claim}
Every log corresponds to a unique sequence $r_{j}$, $1\leqslant j\leqslant M$ of the first $M$ values chosen in the line (3) in some execution of the Algorithm.
\end{claim}

\begin{proof}
For a given log $(R,D,O,P,S)$ we are going to decode $r_{1},\ldots ,r_{M}$.
At first we use information from route $R$ and sequences $D$ and $P$ to
determine which $s_{i}$ were non-zero at each step of the algorithm and to
find coordinates of elements which were zeroed at step (14) of the
Algorithm. Notice that each operation of setting a non-zero value to some $s_{i}$ corresponds to the up-step $(1,1)$ on the route $R$, while each zeroing of $s_{i}$ corresponds to some down-step $(1,-1)$ on route $R$. We examine the route $R$ from the point $(0,0)$ to the point $(2M,0)$. Assume that the
first peak occurs after $j$th step. Since this is the first time we erase some
elements $s_{i}$, we know that $s_{1},\ldots ,s_{j}$ are the only non-zero
elements at this point. Now we use information encoded in $D$ and $P$. We
look at the number of consecutive down-steps on $R$ (which in this case is
equal to $p_{1}$) and remember that for this peak we zeroed $s_{j}$, $%
s_{j-d_{1}}$, $s_{j-2d_{1}},\ldots ,s_{j-(p_{1}-1)d_{1}}$. Then again each
up-step on $R$ denote setting some value to the zeroed position
with the smallest index $i$. Proceeding in that way we know exactly which position was set last, when we reach the next peak. From the number of consecutive down-steps on $R$ we deduce the length of the zeroed repeated block. Value in the sequence $D$ corresponding to the peak denotes the difference of the arithmetic subsequence in which the repetition occured. Finally corresponding value from the sequence $P$ describe the position of the  symbol just set, within the erased repated blocks. From all this information it is easy to deduce which positions was zeroed as a result of erasing the repetition. We repeat these operations
until we get to the end of $R$.

After this preparatory step we are ready to decode $r_{1},\ldots ,r_{M}$. We consider the sequence $R$ in reverse order -- from the
last point $(2M,0)$ to the first $(0,0)$ modifying the final sequence $S$. This time we use information
encoded in $S$ and $O$, and the knowledge determined in a
preparatory step. As we said before, each up-step $(1,1)$ on the route $R$
corresponds to some $r_{i}$. For every such up-step we have already
determined the indices of elements $r_{i}$ on $S$ in the preparatory analysis. At the beginning, going
backward on $R$, there is some number of down-steps corresponding to non-zero
elements of $S$ (the elements added at the end of computations). We skip them and move on. Then, each time there is an
up-step on $R$, we assign to $r_{j}$ a value from appropriate $s_{i}$ (where $i$ was
determined in the preparatory step), and set $s_{i}$ to $0$. In fact, to determine the real outcome of random experiments (i.e. an index of the chosen element on the list of elements available at this step), we must take into account the  forbidden symbols from $k$ preceeding and $k$ following places on $S$. 
Every consecutive sequence of $t$ down-steps on $R$  correspond to erasure of some repeated block during the exectution of the Algorithm. Then we assign
to $s_{i},s_{i+d_{l}},\ldots,s_{i+(t-1)d_{l}}$ corresponding values $%
s_{i+ o_{l}t d_l},s_{i+d_{l}+o_{l}t d_l},\ldots ,s_{i+(t-1)d_{l}+o_{l}t d_l}$ (where $s_i$ is the first element of the ereased repeated block determined in the preparatory step). These are exactly the values from the repetitions erased at step (17) of the 
Algorithm. 
\end{proof}

We showed that each sequence of randomly chosen values during the execution of the Algorithms corresponds to some log, and
that this mapping is injective. This implies that the number of different
logs is always greater than or equal to the number of feasible sequence $r_{1},\ldots ,r_{M}$. Let $L$ be the size of the set of all possible
logs. To calculate $L$ we have to determine the number of different
structures for each element in a log. The number of all possible routes on
the upper right quadrant of a grid of length $2M$ with possible moves $(1,1)$
and $(1,-1)$ is well known to be the $M$th Catalan number $C_M$. Since in every choice in line (3) the elements occurring within the distance $k$ are excluded, the Algorithm can not produce repeated block of length $k$. It means that the subsequence $(1,1),(1,-1),(1,1)$ can not occur in the route $R$. Therefore the number of peaks within $R$ cannot exceed $M/2$.
Thus there can be at most $k^{M/2}$ possible sequences $D$. Respectively, there are at most $2^{M/2}$ possible evaluations for sequence $O$.

The sequence $S$ consists of $n$ elements of value between $0$ and $2 k+10 \sqrt{k}$,
 which gives us $(2 k+10 \sqrt{k})^{n}$ possible evaluations for
this sequence. For every fixed route $R$ with $m$ peaks corresponding to the repeated block of lengths $p_1, \ldots, p_{m}$ we have at most  $p_{1}p_{2}\ldots p_{n}$ sequences which can occurr as $P$.  Therefore for the upper bound for the number sequences $P$ we 
determine maximum value of the product $p_{1}p_{2}\ldots p_{n}$ with $%
p_{1}+\ldots +p_{n}=M$. The inequality between the arithmetic and geometric
means implies that the maximum is obtained when all $p_{i}$ are the same.
Denote their common value by $x$. Then we must determine $\max \left( x^{%
\frac{M}{x}}\right) $. Since

\begin{equation*}
\left(x^{\frac{M}{x}}\right)^{\prime}= x^{\frac{M}{x}}\left(\frac{M}{x^2}-%
\frac{M\log (x)}{x^2}\right),
\end{equation*}

\noindent we get that the maximum value is obtained with $x=e$ and equals $%
\approx 1.44467^{M}<1.5^{M}$.

\noindent All these bounds brings us to the conclusion that the number of
possible logs exceeds

\begin{equation*}
\left(2 k+10 \sqrt{k}\right)^{n}
C_M
k^{M/2}
2^{M/2}
\left(1.5\right)^{M}.
\end{equation*}

\noindent Comparing with the number of evaluations of a sequence $(r_{j})$
we get inequality

\begin{equation*}
(10 \sqrt{k})^M \leq
 \left(2 k+10 \sqrt{k}\right)^{n}
C_M
k^{M/2}
2^{M/2}
\left(1.5\right)^{M}.
\end{equation*}

\noindent Asymptotically, Catalan numbers grow as $C_{n}\sim \frac{4^{n}}{%
n^{3/2}\sqrt{\pi }}$ which implies that

\begin{equation*}
(10 \sqrt{k})^M \leq
\left(2 k+10 \sqrt{k}\right)^{n}
\frac{4^M }{M \sqrt{\pi M}}
k^{M/2}
2^{M/2}
\left(1.5\right)^{M}.
\end{equation*}

The right hand side is $o((10 \sqrt{k})^M )$ therefore for large enough $M$ the inequality can not hold. We get a contradiction, from which we conclude that for some specific choices of $r_1, r_2, \ldots$ the algorithm stops.
\end{proof}

The above proof can be applied in a more general setting.

\begin{theorem}
\label{Theorem 2}Let $K$ be a fixed set of $k$ positive integers. Then for every $n\geqslant 1$ and for any sequence of sets $%
L_{1},\ldots ,L_{n}$ of size at least $2k+10\sqrt{k}$ each, there exists a
sequence $S=s_{1}\ldots s_{n}$ with $s_{i}\in L_{i}$ for all $i=1,2,\ldots ,n$ which is nonrepetitive on every arithmetic progressions whose common difference is in $K$.
\end{theorem}

Note that in the proof of the Theorem 1 we focused only on the number
of forbidden substructures, not their values. Given an arbitrary set of
common differences we order and numerate them from $1$ up to $k$. We can repeat the above reasoning with just one change -- the sequence $D$ consists of elements of $K$ (but there are stil $k$ of them).

\section{A related geometric problem}

As stated in the introduction, the problem of finding sequences
nonrepetitive up to $\bmod \ k$ has its origin in a geometric problem of nonrepetitive coloring of pints in the plane.
We can apply our proof technique to a more general question in this setting. The following problem
concerning nonrepetitive colorings of discrete sets of points in $\mathbb{R}%
^{n}$ was considered in \cite{Grytczuk Rainbow}. Let $P$ be a discrete set
of points and let $L$ be a fixed set of lines in $\mathbb{R}^{n}$. A
coloring of $P$ is \emph{nonrepetitive} (with respect to $L$) if each line in $L$ is colored nonrepetitively (i.e., no sequence
of consecutive points on any $l\in L$ forms a repetition). For a point $%
p\in P$ let $i(p)$ denote the number of lines from $L$ incident with $p$ and
let $I=I(P,L)=\max \{i(p):p\in P\}$ be the maximum incidence of the
configuration $(P,L)$. Using the Lov\'{a}sz Local Lemma it was proved in \cite{Grytczuk Rainbow} that $%
Ie^{(8I^{2}+8I-4)/(I-1)^{2}}$ colors are sufficient to get such a coloring.
Adopting the proof of Theorem 2 we can get a better bound.

\begin{theorem}
Let $(P,L)$ be a configuration of points and lines in $\mathbb{R}^{n}$ with
finite maximum incidence $I>2$. If $C\geqslant 2I+10\sqrt{I}$, then there is
a nonrepetitive $C$-coloring of $P$ with respect to $L$.
\end{theorem}

\begin{proof}
The argument is pretty much the same as in the proof of Theorem 2. We
provide an algorithm for which each point is colored at random by one of $%
2I+10\sqrt{I}$ colors. Fix any linear ordering of all points in $P$. We
color them in this order using Algorithm~1, where arithmetic progressions
are changed into lines in $\mathbb{R}^{n}$. Similarly, for a given point $p \in P$ and every line $l \in L$ such that $p\in l$ we forbid to use colors already assigned to $I$ points preceding and following $p$ on $l$. This gives us at most $2I$ forbidden colors for each point. So, by analogy to the previous proof, one can show that additional $10\sqrt{I}$ colors suffice to get a nonrepetitive coloring of $P$ with respect to $L$. For a log $(R,D,O,P,S)$ we take the same objects as in last case, with the exception that now $D$ keeps the information about the line for which we get a repetition (values between $1$ and $I$), and $S$ is a sequence of numbers between $0$ and $I$. Then all calculations run similarly as before.
\end{proof}

\section{An open problem}

We would like to conclude the paper with a problem concerning infinite sets of forbidden differences. Let $K$ be a fixed (possibly infinite) set of positive integers. A coloring of the integers is $K$\emph{-nonrepetitive} if every arithmetic progression with common difference in $K$ forms a nonrepetitive sequence. Denote by $\pi (K)$ the minimum number of colors (possibly infinite) needed for a $K$-nonrepetitive coloring of $\mathbb{Z}$.

A natural question is for which sets $K$ the number $\pi (K)$ is finite. An obvious necessary condition is that the related integer distance graph (i.e., a graph on the set of vertices $\mathbb{Z}$ with two integers $a>b$ joined by an edge whenever their difference $a-b$ is in $K$) has finite chromatic number, denoted by $\chi (K)$. Theorem 2 shows that $\pi (K)$ is finite for finite sets $K$. More intriguing in this respect is the case of infinite sets $K$. We offer the following conjecture in the spirit of Erd\H{o}s.

\begin{conjecture}
$\pi (K)$ is finite for every lacunary set $K$.
\end{conjecture}

A set $K=\{k_{1}<k_{2}<\ldots \}$ is \emph{lacunary} if there is a real number $\delta >0$ such that $\frac{k_{i+1}}{k_{i}}>1+\delta $ for all indices $i$. For instance the set of powers of $2$ and the set of Fibonacci numbers are lacunary. It is known that for such sets the usual chromatic number $\chi (K)$ is finite \cite{Katznelson}, \cite{PeresSchlag}, \cite{RuzsaTuzaVoigt}. However, there are non-lacunary sets with a finite chromatic number. Complete characterization of such sets is not known and, as pointed out by Ruzsa (personal communication), this problem is connected to some deep questions in additive number theory. A trivial example is the set of odd positive integers, whose chromatic number is $2$. Curiously, for the nonrepetitive variant just $4$ colors suffice as proved by Carpi \cite{Carpi}, which supports even stronger supposition that perhaps $\pi (K)$ is finite if and only if $\chi (K)$ is finite.


\begin{thebibliography}{99}
\bibitem{AlloucheShallit} J-P. Allouche, J. Shallit, Automatic sequences.
Theory, applications, generalizations, Cambridge University Press,
Cambridge, 2003.

\bibitem{AlonSpencer} N. Alon, J.H. Spencer, The probabilistic method,
Second Edition, John Wiley \& Sons, Inc., New York, 2000.

\bibitem{Beanetal} D. R. Bean, A. Ehrenfeucht, G. F. McNulty, Avoidable
patterns in strings of symbols, Pacific J. Math. 85 (1979), 261-294.

\bibitem{BerstelWork} J. Berstel, Axel Thue's work on repetitions in words;
in P. Leroux, C. Reutenauer (eds.), S\'{e}ries formelles et combinatoire alg%
\'{e}brique Publications du LaCIM,, Universit\'{e} du Qu\'{e}bec a Montr\'{e}%
al, p 65-80, 1992.

\bibitem{BerstelTrans} J. Berstel, Axel Thue's papers on repetitions in
words: a translation, Publications du LaCIM, vol 20, Universit\'{e} du Qu%
\'{e}bec a Montr\'{e}al, 1995.

\bibitem{Carpi} A. Carpi, Multidimensional unrepetitive configurations, Theoret. Comput. Sci. 56 (1988) 233241.

\bibitem{CurrieSimpsson} {J. Currie and J. Simpson}, Non-repetitive Tilings, 
{The Electron. J. Comb.}, 9 (2002), {2--8}.

\bibitem{CurriePattern} J. Currie, Pattern avoidance; themes and variations,
Theor. Comput. Sci., 339 (2005), 7--18.

\bibitem{Grytczuk Rainbow} {J.Grytczuk}, Thue-Like Sequences and Rainbow
Arithmetic Progressions, {The Electr. J. Comb.}, 9(1) (2002), Research Paper
44, 10.

\bibitem{GrytczukIJMMS} J. Grytczuk, Nonrepetitive colorings of graphs - a
survey, Int. J. Math. Math. Sci. (2007), Art. ID 74639, 10 pp.

\bibitem{GrytczukMiki} J. Grytczuk, Thue type problems for graphs, points,
and numbers, Discrete Math. 308 (2008) 4419--4429.

\bibitem{GrytczukKozikMicek} J. Grytczuk, J. Kozik, P. Micek, A new approach
to nonrepetitive sequences (submitted).

\bibitem{Katznelson} Y. Katznelson, Chromatic numbers of Cayley graphs on $\mathbb{Z}$ and recurrence, Combinatorica 21 (2001), 211219.

\bibitem{Lothaire1} M. Lothaire, Combinatorics on Words, Addison-Wesley,
Reading MA, 1983.

\bibitem{MoserTardos}  R. Moser, G. Tardos, A constructive proof of the
general Lov\'{a}sz local lemma, {J. ACM}, 57 (2010), Art. 11, 15.

\bibitem{PeresSchlag} Y. Peres, W. Schlag, Two Erd\H{o}s problems on lacunary sequences: Chromatic number and Diophantine approximation, Bull. London Math. Soc. 42 (2010) 295300.

\bibitem{RuzsaTuzaVoigt} I. Ruzsa, Z. Tuza, M. Voigt, Distance graphs with finite chromatic number, J. Combin. Theory, Ser. B, 85 (2002), 181--187.

\bibitem{Thue1} A. Thue, \"{U}ber unendliche Zeichenreichen, Norske Vid.
Selsk. Skr., I Mat. Nat. Kl., Christiania, 7 (1906), 1-22.
\end{thebibliography}
\end{document}